\documentclass[]{theclass}
\usepackage{mathtools}
\usepackage{enumitem} 
\usepackage{tikz}

\newcommand{\packA}{\ensuremath{(1,2,2,2,2)}}
\newcommand{\packB}{\ensuremath{(1,1,3,3,4)}}
\newcommand{\col}{\varphi}
\newcommand{\pos}{\{1,2,3,4\}}
\newcommand{\PP}{\mathcal P}
\newcommand{\ind}{\mathbf 1}

\definecolor{orchid}{rgb}{0.85, 0.44, 0.84}

\begin{document}
%\linenumbers

\begin{frontmatter}

\titledata{Packing colourings of claw-free subcubic graphs}{}

\authordata{Federico Romaniello}
{Dipartimento per l'Innovazione Umanistica, Scientifica e Sociale\\ Universit\`{a} degli Studi della Basilicata, Potenza, Italy\\ }{federico.romaniello@unibas.it}{}
\end{frontmatter}
\begin{abstract}
We prove that every claw-free subcubic graph admits a \packA-packing colouring, confirming a conjecture of Mortada, El Zein and Al Hajjar. The proof combines a reduction to line graphs of subdivided
loopless cubic multigraphs with local colouring extensions. We also give a negative answer to their question on \packB-packing colourings by constructing an infinite family of connected bridgeless claw-free cubic counterexamples. The smallest graph in this family has 36 vertices. More generally, we characterise the simple cubic graphs whose triangle expansions admit a \packB-packing colouring in terms of gated independent odd-cycle transversals. This characterisation gives a classification for connected bridgeless claw-free cubic bases and shows that the double triangle expansion of a loopless cubic multigraph admits such a colouring if and only if the multigraph has a perfect matching.
\end{abstract}

\section{Introduction}\label{sec:introduction}
All graphs considered are finite and undirected. Graphs are simple unless explicitly described as multigraphs. Multigraphs may have parallel edges and, when stated, loops; a loop contributes two to the degree. A graph is \emph{subcubic} if its maximum degree is at most three, \emph{cubic} if every vertex has degree three, and \emph{claw-free} if it has no induced copy of $K_{1,3}$.

For a positive integer $i$, an \emph{$i$-packing} in a graph $G$ is a set of vertices whose pairwise distances in $G$ are greater than $i$. Given a nondecreasing sequence $S=(s_1,\ldots,s_t)$ of positive integers, an \emph{$S$-packing colouring} of $G$ is a partition of $V(G)$ into sets $V_1,\ldots,V_t$ such that $V_i$ is an $s_i$-packing for every $i\in\{1,\ldots,t\}$. Empty classes are permitted. In particular, a $1$-packing is an independent set, and a $(1,\ldots,1)$-packing colouring is an ordinary proper colouring.

When $S=(1,2,\ldots,t)$, this definition gives the packing colouring introduced by Goddard et al.~\cite{GHHR} under the name of broadcast colouring. The more general notion of $S$-packing colouring was studied by Goddard and Xu~\cite{GX}. Allowing repeated entries in $S$ makes it possible to combine independent sets with colour classes satisfying larger distance requirements.

Gastineau and Togni~\cite{GT} proved that every subcubic graph admits both a $(1,1,2,2,2)$-packing colouring and a $(1,2,2,2,2,2,2)$-packing colouring, and posed several questions concerning stronger distance requirements. For claw-free cubic graphs, Bre\v{s}ar, Kuenzel and Rall~\cite{BKR} established $(1,1,2,2)$-packing colourability.

El Zein and Mortada~\cite{EM} proved that every subcubic graph in which each vertex of degree three lies on a cycle of length at most four admits a $(1,2,2,2,2,2)$-packing colouring. They asked whether the number of $2$-packing classes can be reduced from five to four. Every vertex of degree three in a claw-free graph lies on a triangle. Our first result therefore answers their question for claw-free subcubic graphs.

Recently, Mortada, El Zein and Al Hajjar~\cite{MEA} proved that every claw-free subcubic graph admits a $(1,1,2,3)$-packing colouring. They also shown that every connected claw-free subcubic graph, apart from a specified graph $\mathcal H$ on 12 vertices, admits a $(1,1,3,3,3)$-packing colouring. The graph $\mathcal H$ is the triangle expansion $T(K_4)$; the operation $T$ is defined in Section~\ref{sec:preliminaries}.

In this paper we consider two questions proposed in~\cite{MEA}. Their Conjecture~1 states that every claw-free subcubic graph admits a \packA-packing colouring. Their Problem~1 asks whether every connected claw-free subcubic graph distinct from $\mathcal H$ admits a \packB-packing colouring. Thus the first question asks for a partition into one independent set and four $2$-packings, whereas the second asks for two independent sets, two $3$-packings and one $4$-packing.

We first prove the following theorem, which confirms Conjecture~1 of~\cite{MEA}.

\begin{theorem}\label{thm:mainA}
Let $G$ be a claw-free subcubic graph. Then $G$ admits a \packA-packing colouring.
\end{theorem}

We then give a negative answer to Problem~1 of~\cite{MEA}. Let $R_k$ denote the ring of $k$ diamonds defined in Section~\ref{sec:negative}.

\begin{theorem}\label{thm:mainB}
Let $k\ge3$ be an odd integer. Then $T(R_k)$ is a connected bridgeless claw-free and diamond-free cubic graph with $12k$ vertices and $18k$ edges, and it admits no \packB-packing colouring.
\end{theorem}

For the positive result, we apply the packing edge-colouring theorem of Yang and Wu~\cite{YW} to subdivisions of loopless cubic multigraphs. This gives the required colouring of their triangle expansions. A structural decomposition and local extensions then yield Theorem~\ref{thm:mainA}, without any assumption on bridges. For the negative result, we prove that the triangle expansion of a simple cubic graph admits a \packB-packing colouring if and only if the base graph has a gated independent odd-cycle transversal. This criterion yields Theorem~\ref{thm:mainB}, a classification for connected bridgeless claw-free cubic bases, and a characterisation of colourable double triangle expansions in terms of perfect matchings.

Section~\ref{sec:preliminaries} contains the structural preliminaries. The two main theorems are proved in Sections~\ref{sec:positive} and~\ref{sec:negative}, respectively. Section~\ref{sec:classification} gives the further characterisations, and Section~\ref{sec:questions} discusses the remaining questions.

\section{Preliminaries}\label{sec:preliminaries}
For a graph $G$, let $N_G(v)$ denote the open neighbourhood of $v$, and put $N_G(X)=\bigcup_{v\in X}N_G(v)$.

A \emph{diamond} has vertices $a,b,c,d$ and edges $ab,ac,bc,bd,cd$. Its nonadjacent vertices $a,d$ are its \emph{terminals}, and $b,c$ are its \emph{centres}. A \emph{string of diamonds} is a sequence of vertex-disjoint diamonds joined consecutively through their terminals. To replace an edge $uv$ by such a string, delete $uv$ and join $u,v$ to the two free terminals. A \emph{ring of diamonds} joins such diamonds cyclically through their terminals; the exact notation $R_k$ is given in Section~\ref{sec:negative}.

Let $H$ be a loopless cubic multigraph. The vertices of its \emph{triangle expansion} $T(H)$ are the incidences $(v,e)$ with $v\in V(H)$ and $e$ incident with $v$. For each $v$, the three incidences at $v$ induce a triangle, denoted by $Q_v$. For each edge $e=uv$, add the edge $(u,e)(v,e)$. Edges inside $Q_v$ are called \emph{internal}; the other edges are \emph{external}. The external edges form a perfect matching. Distinct incidences are used even for parallel edges. If $H$ is simple, we also write $(v,u)$ for the incidence corresponding to $vu$.

\begin{lemma}\label{common:expansion}
Let $H$ be a loopless cubic multigraph. Then $T(H)$ is simple, cubic, claw-free and diamond-free. It is connected whenever $H$ is connected. An external edge of $T(H)$ is a bridge if and only if its corresponding edge in $H$ is a bridge; no internal edge is a bridge.
\end{lemma}
\begin{proof}
The use of distinct incidence vertices prevents loops and parallel edges. Each vertex has two internal neighbours and one external neighbour, and its internal neighbours are adjacent. Moreover, the only triangles are the $Q_v$: a triangle using an external edge would need two consecutive external edges, which is impossible because they form a matching. The triangles are disjoint, so no diamond occurs. Connectivity is preserved. Internal edges lie on triangles. Every cycle through an edge of $H$ lifts to a cycle through its corresponding external edge; conversely, a bridge of $H$ separates the corresponding two collections of triangles in $T(H)$.
\end{proof}

We use the structural description of Gutierrez and Ugarte~\cite[Proposition~3]{GU}, which extends Oum's bridgeless description~\cite[Proposition~1]{Oum} to all connected claw-free cubic graphs. 

\begin{lemma}\label{common:structure}
Let $F$ be a connected claw-free cubic graph. Then either
\begin{enumerate}[label=(\roman*)]
\item $F\cong K_4$;
\item $F$ is a ring of at least two diamonds; or
\item $F$ is obtained from a connected cubic multigraph $H$ by replacing every vertex with a triangle and replacing edges with possibly empty strings of diamonds. Every loop of $H$ is replaced by a nonempty string.
\end{enumerate}
\end{lemma}

\section{\texorpdfstring{$\packA$}{(1,2,2,2,2)}-packing colourings of claw-free subcubic graphs}\label{sec:positive}
Throughout this section we use the colour set $\{0,1,2,3,4\}$. Colour $0$ is the independent-set colour, and the colours in $\pos$ are called \emph{positive colours}. A colouring $\col$ is valid if adjacent vertices have different colours and distinct vertices of the same positive colour have distance at least three. Positive colours may be permuted freely; colour $0$ is kept fixed.

\subsection{Triangle expansions and local extensions}
\begin{lemma}\label{pos:lem:expansion}
Let $H$ be a loopless cubic multigraph. Then $T(H)$ admits a \packA-packing colouring.
\end{lemma}
\begin{proof}
Let $S(H)$ be obtained by subdividing each edge $e$ of $H$ once, and denote the new vertex by $s_e$. Since $H$ is loopless and distinct edges receive distinct subdivision vertices, $S(H)$ is simple, even when $H$ has parallel edges. Every edge of $S(H)$ joins a vertex of degree three to a vertex of degree two.

Yang and Wu~\cite{YW} proved that every graph in which each edge has endpoint degree sum at most five admits a \packA-packing edge-colouring, with distances measured in the line graph. Applying their theorem to $S(H)$ gives a \packA-packing colouring of its line graph $L(S(H))$.

The map $(v,e)\mapsto vs_e$ is an isomorphism from $T(H)$ to $L(S(H))$: the three edges incident with an original vertex $v$ form its replacing triangle, and the two edges incident with $s_e$ give the external edge corresponding to $e$. Thus the colouring transfers to $T(H)$.
\end{proof}

We will repeatedly replace an edge by a graph with two attachment vertices, or replace a subgraph attached through a single edge. These operations do not shorten distances between vertices of the retained part. To check an extension across one attachment edge $wx$, it is enough to check the new colours against $w$ and its retained neighbours: only $x$ can be at distance two from a retained neighbour of $w$. For a two-ended replacement of an edge $uv$, the corresponding checks are made at $u,v$ and their retained neighbours. We make these checks explicit below.

\begin{lemma}\label{pos:lem:diamond}
Let $F$ be a graph with a \packA-packing colouring $\col$, and let $uv\in E(F)$. Let $F'$ be obtained by deleting $uv$, adding a diamond with terminals $a,d$, and adding the edges $ua,dv$. Then $\col$ extends to a \packA-packing colouring of $F'$.
\end{lemma}
\begin{proof}
Write $x=\col(u)$ and $y=\col(v)$; necessarily $x\ne y$. Replace $uv$ by the diamond $a,b,c,d$ and the links $ua,dv$. Set
\[
\col(a)=y,\qquad \col(d)=x,
\]
and give $b,c$ two distinct positive colours outside $\{x,y\}$. At least two such positive colours are available.

The assignment is valid inside the diamond. If $\col(a)\ne0$, its colour was the old colour of $v$, so it differs from the colours of $u$ and all retained neighbours of $u$; these were at distance at most two from $v$ before the replacement. If $\col(a)=0$, the required inequality with $u$ follows from $x\ne y$. The same argument applies to $d$ at $v$. The vertices $b,c$ are at distance two from both attachment vertices and their colours avoid $x,y$; they are at distance at least three from all other old vertices. Distances between old vertices cannot decrease. Therefore the entire colouring is valid.
\end{proof}

The next local extension will handle loops in a representing multigraph. Let $L$ consist of a triangle $z_0z_1z_2$, a diamond $a,b,c,d$, and the edges $z_0a,dz_1$. The vertex $z_2$ has degree two and every other vertex of $L$ has degree three. Attaching $L$ at a vertex $w$ means adding $wz_2$; see Figure~\ref{pos:fig:lobe}.

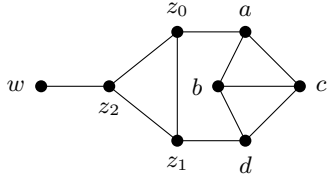
\begin{figure}[tb]
\centering
\begin{tikzpicture}[scale=.9,v/.style={circle,fill=black,inner sep=1.6pt},every label/.style={font=\small}]
\coordinate (w) at (-2,0);\coordinate (z2) at (-1,0);\coordinate (z0) at (0,.8);\coordinate (z1) at (0,-.8);
\coordinate (a) at (1,.8);\coordinate (d) at (1,-.8);\coordinate (b) at (.6,0);\coordinate (c) at (1.8,0);
\draw (w)--(z2)--(z0)--(z1)--(z2) (z0)--(a)--(b)--(c)--(a) (b)--(d)--(c) (d)--(z1);
\node[v,label=left:$w$] at (w) {};
\node[v,label=below:$z_2$] at (z2) {};
\node[v,label=above:$z_0$] at (z0) {};
\node[v,label=below:$z_1$] at (z1) {};
\node[v,label=above:$a$] at (a) {};
\node[v,label=below:$d$] at (d) {};
\node[v,label=left:$b$] at (b) {};
\node[v,label=right:$c$] at (c) {};
\end{tikzpicture}
\caption{The pendant loop lobe $L$ attached at the boundary vertex $w$ of the retained graph.}\label{pos:fig:lobe}
\end{figure}

\begin{lemma}\label{pos:lem:loop}
Let $F$ be a graph with a \packA-packing colouring $\col$, and let $w\in V(F)$ have degree at most two. Let $F'$ be obtained by attaching a copy of $L$ at $w$. Then $\col$ extends to a \packA-packing colouring of $F'$.
\end{lemma}
\begin{proof}
If $\col(w)\ne0$, we may assume that $\col(w)=1$. If $\col(w)=0$, at most two positive colours occur on its neighbours; choose an absent one and call it $1$. Apply Table~\ref{pos:tab:loop}.

\begin{table}[htb]
\centering
\begin{tabular}{ccl}
\toprule
$\col(w)$ & Additional convention & $(\col(z_0),\col(z_1),\col(z_2),\col(a),\col(b),\col(c),\col(d))$\\
\midrule
1 & none & $(2,3,0,0,1,4,0)$\\
0 & $1\notin\col(N_F(w))$ & $(0,2,1,3,1,4,0)$\\
\bottomrule
\end{tabular}
\caption{The seven-vertex pendant loop lobe.}\label{pos:tab:loop}
\end{table}

In the first row all positive colours in $L$ are distinct and the zeros are independent. In the second row the only repeated positive colour is $1$, on $z_2,b$, and these vertices have distance three; again the zeros are independent. If $\col(w)=1$, the attachment vertex $z_2$ receives colour $0$ and its internal neighbours avoid $1$. If $\col(w)=0$, the attachment vertex $z_2$ has the positive colour $1$ absent from $N_F(w)$. The two internal neighbours of $z_2$ are at distance two from $w$, which creates no additional restriction when $\col(w)=0$. No other new vertex lies within distance two of $F$. Since the attachment is pendant, distances between old vertices are unchanged.
\end{proof}

\subsection{Claw-free cubic graphs}
\begin{theorem}\label{pos:thm:cubic}
Let $G$ be a claw-free cubic graph. Then $G$ admits a \packA-packing colouring.
\end{theorem}
\begin{proof}
We may work with a connected component and apply Lemma~\ref{common:structure}. For $K_4$, assign the distinct colours $0,1,2,3$.

For a ring of $k\ge2$ diamonds, colour one centre of each diamond $0$. The remaining vertices form a cycle of length $3k$; colour them cyclically $1,2,3,1,2,3,\ldots$. Any two remaining vertices at distance two through a deleted centre were already at distance at most two on this cycle, because the three neighbours of that centre form a three-vertex path on it. Thus the colouring is valid in the whole ring.

Consider now a representation as in Lemma~\ref{common:structure}(iii). We induct on the number $\ell$ of loops in its cubic multigraph $H$. If $\ell=0$, Lemma~\ref{pos:lem:expansion} colours $T(H)$; repeated applications of Lemma~\ref{pos:lem:diamond} insert all required strings.

Suppose $\ell>0$, and choose a loop at $x$. The triangle corresponding to $x$, together with the nonempty string replacing this loop, forms a pendant lobe. It is joined to the rest by a single edge, whose outside endpoint we call $w$. The vertex $w$ may be in a triangle or may be a terminal of another diamond.

Let $J$ be the five-vertex graph obtained by subdividing one edge of a fresh $K_4$. Its root $r$ has degree two, and its other vertices have degree three. Replace the entire chosen lobe by its expanded one-pole: replace each of the five vertices of $J$ with a triangle, using the third vertex of the root triangle as a free attachment vertex, and join that vertex to $w$. The resulting graph $F'$ is simple, cubic and claw-free. In its representation, the base vertex $x$ and its loop have been replaced by the five vertices and simple edges of $J$, while the third incident connection is retained at $r$. Hence the new representing multigraph has exactly $\ell-1$ loops.

By induction, colour $F'$. Delete the expanded one-pole and retain all outside colours. The outside vertex $w$ now has degree two. Lemma~\ref{pos:lem:loop} attaches and colours the seven-vertex lobe $L$, corresponding to a loop replaced by one diamond. If the original loop string had more diamonds, insert the additional diamonds on a link of this string using Lemma~\ref{pos:lem:diamond}. This restores the original graph without changing the outside colouring and completes the induction.
\end{proof}

\subsection{Subdivision and the subcubic case}
We next prove a subdivision lemma that will be used to pass from cubic to subcubic graphs.

\begin{lemma}\label{pos:lem:subdivision}
Let $F$ be a simple subcubic graph with a \packA-packing colouring, and let $uv\in E(F)$. Assume that at each endpoint of degree three, its two neighbours other than the opposite endpoint are adjacent. Let $F'$ be obtained by replacing $uv$ with the path $uwv$, where $w$ is new. Then $F'$ admits a \packA-packing colouring.
\end{lemma}
\begin{proof}
Subdividing an edge cannot decrease distances between old vertices. If both $u$ and $v$ have positive colours, set $\col(w)=0$ and retain every old colour.

Otherwise, exactly one endpoint has colour $0$. By reversing the edge and permuting positive colours, assume
\[
\col(u)=0,\qquad \col(v)=1.
\]
We set $\col(w)=0$ and recolour $u$, sometimes also one neighbour of $u$.

If $d_F(u)=1$, give $u$ colour $2$. Its only old vertex at distance two in $F'$ is $v$, whose colour is $1$.

Suppose $d_F(u)=2$, and let $a$ be the neighbour other than $v$. Since $a$ is adjacent to a zero vertex, it is positive; moreover $a$ and $v$ had distance two, so we may call its colour $2$. If a positive colour is absent from
\[
\{1,2\}\cup \col\bigl(N_F(a)\setminus\{u\}\bigr),
\]
give that colour to $u$. This checks every old vertex within distance two of $u$ in $F'$. If no positive colour is absent, then $a$ has two other neighbours with colours $3$ and $4$. Set
\[
\col'(u)=2,\qquad \col'(a)=0.
\]
The neighbours of $a$ now have colours $2,3,4$, so the new zero causes no adjacency conflict. Every vertex at distance one or two from $u$ has colour in $\{0,1,3,4\}$, proving that its new positive colour is valid.

Finally suppose $d_F(u)=3$, with other neighbours $a,b$. By hypothesis $ab\in E(F)$. The colours of $a,b$ are positive and distinct, and both differ from $\col(v)=1$; call them $2,3$, respectively. Each of $a,b$ has at most one remaining neighbour, denoted $a'$ or $b'$ when present. Such a neighbour is within distance two of both $a$ and $b$, so its colour belongs to $\{0,1,4\}$. These neighbours may coincide with each other or with $v$.

If neither remaining neighbour has colour $4$, give $u$ colour $4$. Its new distance-two neighbourhood contains only colours $0,1,2,3$. Otherwise, by interchanging $a,b$ and the names $2,3$, assume that $a'$ has colour $4$. Set
\[
\col'(u)=2,\qquad \col'(a)=0.
\]
The neighbours of $a$ are now $u,b,a'$, with colours $2,3,4$, so its new zero is valid. The vertices at distance one or two from $u$ lie among
\[
\{w,a,b,v,a',b'\},
\]
with colours in $\{0,1,3,4\}$ after the recolouring. Thus $u$ has no positive-colour conflict. In every case all other colours are unchanged; the new zero $w$ has positive neighbours $u,v$, and old distances do not decrease. This proves the lemma.
\end{proof}

We can now prove the main theorem.

\begin{proof}[Proof of Theorem~\ref{thm:mainA}]
Suppose that a counterexample exists and choose one, $G$, with the minimum number of vertices. Components can be coloured independently, so $G$ is connected.

An isolated vertex can receive colour $0$. A leaf $x$ can also be restored to a colouring of $G-x$: if its neighbour $u$ has a positive colour, give $x$ colour $0$; if $\col(u)=0$, give $x$ a positive colour absent from the at most two other neighbours of $u$. These are the only old vertices at distance two from $x$. Hence
\[
\delta(G)\ge2.
\]

Suppose that $x$ is a degree-two vertex not lying in a triangle. Its neighbours $u,v$ are nonadjacent. Suppress $x$, replacing $ux,xv$ with the new edge $uv$, to obtain a smaller simple subcubic graph $G'$. We claim that $G'$ is claw-free. Only $u,v$ change their sets of neighbours, and adding $uv$ cannot create a claw centred elsewhere. If $u$ has degree three, let $a,b$ be its other neighbours. The vertex $x$ was adjacent to neither $a$ nor $b$, because it lies in no triangle. Since $G$ is claw-free, $ab$ must be an edge. Therefore $u$ still cannot be the centre of a claw in $G'$. The same argument applies at $v$; endpoints of degree two cannot be claw centres.

By minimality, $G'$ is colourable. The preceding argument also verifies the endpoint hypothesis of Lemma~\ref{pos:lem:subdivision} for its edge $uv$. Subdividing that edge recovers $G$ and gives a valid colouring, a contradiction. Consequently, every degree-two vertex of $G$ lies in a triangle.

For each degree-two vertex $x$ of $G$, take a fresh copy of the seven-vertex graph $L$ from Lemma~\ref{pos:lem:loop} and join $x$ to its degree-two vertex $z_2$. Denote the resulting graph by $\widehat G$. It is simple and cubic. It is also claw-free: at each old attachment vertex $x$, its two old neighbours are adjacent because $x$ lies in a triangle; at the new attachment vertex $z_2$, its two neighbours in $L$ are adjacent; every other neighbourhood is unchanged and was already claw-free.

By Theorem~\ref{pos:thm:cubic}, $\widehat G$ has a \packA-packing colouring. Restrict it to $G$. Deleting vertices cannot decrease distances between the retained vertices, so the restricted colouring remains valid. This final contradiction proves the theorem.
\end{proof}

\section{\texorpdfstring{$\packB$}{(1,1,3,3,4)}-packing colourings of triangle expansions}\label{sec:negative}
\subsection{A characterisation in terms of the base graph}
An \emph{independent odd-cycle transversal} is an independent set $C$ such that $G-C$ is bipartite. A neighbour $x$ of $c\in C$ is \emph{private with respect to $C$} if
\[
N_G(x)\cap C=\{c\}.
\]
Following Civan, Deniz and Yetim~\cite{CDY}, an independent set with a private neighbour for each of its vertices is called \emph{gated}. We write $\PP(G)$ for the existence of a gated independent odd-cycle transversal in $G$.

The selection of $3$-packing representatives used below is a special case of the Hall-type construction in Mortada, El Zein and Al Hajjar~\cite[Lemma~4]{MEA}. For triangle expansions, the following lemma gives an exact distance criterion and identifies the additional private-neighbour condition required for a $4$-packing.

\begin{lemma}\label{neg:lem:representatives}
Let $G$ be a simple cubic graph and let $I$ be an independent set of $G$. Choose $f(v)\in N_G(v)$ for every $v\in I$, and let $p_v=(v,f(v))$ be the corresponding vertex of $Q_v$ in $T(G)$.
\begin{enumerate}[label=(\roman*)]
\item The set $\{p_v:v\in I\}$ is a $3$-packing if and only if $f$ is injective.
\item The set $\{p_v:v\in I\}$ is a $4$-packing if and only if $N_G(f(v))\cap I=\{v\}$ for every $v\in I$.
\end{enumerate}
\end{lemma}
\begin{proof}
A path between different replacing triangles that uses $r$ external edges has length at least $2r-1$, since no two external edges are consecutive. Its projection to the base is a walk between the corresponding base vertices. Hence
\[
d_{T(G)}(p_u,p_v)\ge 2d_G(u,v)-1.
\]
In particular, base distance at least three guarantees expanded distance at least five.

It remains to consider distinct $u,v\in I$ with a common neighbour. For every $x\in N_G(u)\cap N_G(v)$, the path through $Q_x$ has length
\[
3+\ind_{f(u)\ne x}+\ind_{f(v)\ne x}.
\]
These lengths are at most five, whereas a path using three or more external edges has length at least five. Therefore
\begin{equation*}\label{neg:eq:distance}
d_{T(G)}(p_u,p_v)=\min_{x\in N_G(u)\cap N_G(v)}
\bigl(3+\ind_{f(u)\ne x}+\ind_{f(v)\ne x}\bigr).
\end{equation*}
The distance is at least four exactly when $f(u)\ne f(v)$. It is at least five exactly when neither $f(u)$ nor $f(v)$ is a common neighbour of $u$ and $v$. Applying these observations to every pair in $I$ proves both statements.
\end{proof}

\begin{theorem}\label{neg:thm:reduction}
Let $G$ be a simple cubic graph. Then $T(G)$ admits a \packB-packing colouring if and only if $\PP(G)$ holds.
\end{theorem}
\begin{proof}
Suppose first that $T(G)$ has such a colouring, with colours $1_a,1_b,3_a,3_b,4$. Every triangle $Q_v$ contains a vertex with a colour in $\{3_a,3_b,4\}$. Retain one such vertex in each $Q_v$ and erase the colours of every other vertex. The graph induced by the uncoloured vertices is the union of two matchings: one consists of the remaining internal edges, and the other consists of the remaining external edges. It is therefore bipartite and can be coloured with $1_a,1_b$. Thus we may assume that every $Q_v$ contains exactly one vertex coloured $3_a$, $3_b$ or $4$.

Partition $V(G)$ into $A,B,C$ according to these three colours. Any two vertices in triangles corresponding to adjacent base vertices are at distance at most three. Consequently, $A,B,C$ are independent, and $G-C$ is bipartite. For each $c\in C$, let $f(c)$ be the neighbour indicated by its representative in $Q_c$. By Lemma~\ref{neg:lem:representatives}(ii), $N_G(f(c))\cap C=\{c\}$. Thus $C$ is gated.

Conversely, let $C$ be a gated independent odd-cycle transversal, and let $(A,B)$ be a bipartition of $G-C$. For either $I=A$ or $I=B$, consider the bipartite graph with parts $I$ and $N_G(I)$, using the edges inherited from $G$. For every $X\subseteq I$, cubicity gives
\[
3|X|=|E_G(X,N_G(X))|\le 3|N_G(X)|.
\]
Hall's theorem supplies an injective map $f_I:I\to N_G(I)$ such that $f_I(v)\in N_G(v)$. Colour the corresponding representatives $3_a$ when $I=A$ and $3_b$ when $I=B$. For each $c\in C$, choose a private neighbour $f_C(c)$ and colour $(c,f_C(c))$ with $4$. Lemma~\ref{neg:lem:representatives} verifies the required packing distances.

Exactly one vertex has been selected in each $Q_v$. The unselected vertices again induce the union of the internal and external matchings, so they can be coloured with $1_a,1_b$.
\end{proof}

\subsection{Odd rings of diamonds}
For $k\ge2$, let $R_k$ be the graph with vertices $a_i,b_i,c_i,d_i$ for $i\in\mathbb Z/k\mathbb Z$, and edges
\[
a_ib_i,\quad a_ic_i,\quad b_ic_i,\quad b_id_i,\quad c_id_i,
\quad d_i a_{i+1}
\qquad(i\in\mathbb Z/k\mathbb Z).
\]
The four vertices with a fixed index induce a diamond $D_i$. We call $a_i,d_i$ its \emph{terminals} and $b_i,c_i$ its \emph{centres}. The graph $R_k$ is simple, connected, cubic, claw-free and bridgeless, and has $4k$ vertices. Indeed, every edge internal to a diamond lies on a triangle, while the linking edges lie on the cycle obtained by following $a_i b_i d_i$ in each diamond.

\begin{theorem}\label{neg:thm:rings}
Let $k\ge2$. Then $\PP(R_k)$ holds if and only if $k$ is even. Equivalently, $T(R_k)$ admits a \packB-packing colouring if and only if $k$ is even.
\end{theorem}
\begin{proof}
Let $C$ be an independent odd-cycle transversal of $R_k$. Since $C$ meets both triangles $a_i b_i c_i$ and $b_i c_i d_i$, its intersection with $D_i$ has one of precisely three forms:
\[
\{b_i\},\qquad \{c_i\},\qquad \{a_i,d_i\}.
\]
Suppose that $C\cap V(D_i)=\{a_i,d_i\}$. The internal neighbours $b_i,c_i$ of $a_i$ each see both selected terminals, so neither is private. Its external neighbour is $d_{i-1}$. This vertex cannot belong to $C$, by independence. Hence the preceding diamond has a selected centre, and $d_{i-1}$ is adjacent both to that centre and to $a_i$. It is not private either. Thus $a_i$ has no private neighbour, and $C$ is not gated.

It follows that a gated independent odd-cycle transversal must select one centre in every diamond. Its complement is exactly the cycle $C_{3k}$, which is bipartite only when $k$ is even. Conversely, if $k$ is even, take $C=\{b_i:i\in\mathbb Z/k\mathbb Z\}$. This set is independent, its complement is $C_{3k}$, and $c_i$ is a private neighbour of $b_i$ for every $i$. The final equivalence follows from Theorem~\ref{neg:thm:reduction}.
\end{proof}

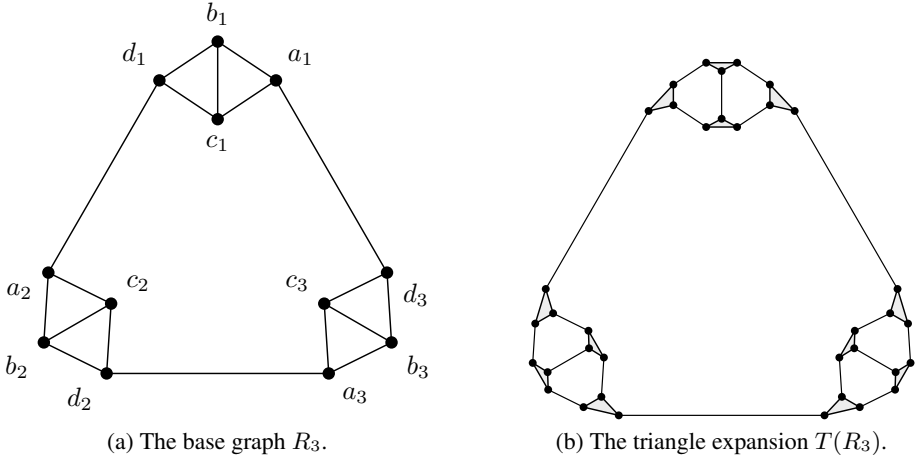
\begin{figure}[tb]
\centering
\begin{minipage}[b]{0.46\textwidth}
\centering
\begin{tikzpicture}[scale=1.07,vertex/.style={circle,fill=black,inner sep=1.65pt}]
\coordinate (a0) at (0.7200,2.0000);
\coordinate (b0) at (0.0000,2.4800);
\coordinate (c0) at (0.0000,1.5200);
\coordinate (d0) at (-0.7200,2.0000);
\coordinate (a1) at (1.3721,-1.6235);
\coordinate (b1) at (2.1477,-1.2400);
\coordinate (c1) at (1.3164,-0.7600);
\coordinate (d1) at (2.0921,-0.3765);
\coordinate (a2) at (-2.0921,-0.3765);
\coordinate (b2) at (-2.1477,-1.2400);
\coordinate (c2) at (-1.3164,-0.7600);
\coordinate (d2) at (-1.3721,-1.6235);
\draw[line width=.55pt] (a0) -- (b0);
\draw[line width=.55pt] (a0) -- (c0);
\draw[line width=.55pt] (b0) -- (c0);
\draw[line width=.55pt] (b0) -- (d0);
\draw[line width=.55pt] (c0) -- (d0);
\draw[line width=.55pt] (d0) -- (a2);
\draw[line width=.55pt] (a1) -- (b1);
\draw[line width=.55pt] (a1) -- (c1);
\draw[line width=.55pt] (b1) -- (c1);
\draw[line width=.55pt] (b1) -- (d1);
\draw[line width=.55pt] (c1) -- (d1);
\draw[line width=.55pt] (d1) -- (a0);
\draw[line width=.55pt] (a2) -- (b2);
\draw[line width=.55pt] (a2) -- (c2);
\draw[line width=.55pt] (b2) -- (c2);
\draw[line width=.55pt] (b2) -- (d2);
\draw[line width=.55pt] (c2) -- (d2);
\draw[line width=.55pt] (d2) -- (a1);
\node[vertex,label=70.2:{$ a_1 $}] at (a0) {};
\node[vertex,label=90.0:{$ b_1 $}] at (b0) {};
\node[vertex,label=270.0:{$ c_1 $}] at (c0) {};
\node[vertex,label=109.8:{$ d_1 $}] at (d0) {};
\node[vertex,label=-49.8:{$ a_3 $}] at (a1) {};
\node[vertex,label=-30.0:{$ b_3 $}] at (b1) {};
\node[vertex,label=150.0:{$ c_3 $}] at (c1) {};
\node[vertex,label=-10.2:{$ d_3 $}] at (d1) {};
\node[vertex,label=-169.8:{$ a_2 $}] at (a2) {};
\node[vertex,label=-150.0:{$ b_2 $}] at (b2) {};
\node[vertex,label=30.0:{$ c_2 $}] at (c2) {};
\node[vertex,label=-130.2:{$ d_2 $}] at (d2) {};
\end{tikzpicture}
\par\small (a) The base graph $R_3$.
\end{minipage}\hfill
\begin{minipage}[b]{0.49\textwidth}
\centering
\begin{tikzpicture}[scale=1.17,vertex/.style={circle,fill=black,inner sep=1.05pt}]
\coordinate (a0xb0) at (0.5453,2.1165);
\coordinate (a0xc0) at (0.5453,1.8835);
\coordinate (a0xd1) at (0.8250,1.8181);
\coordinate (b0xa0) at (0.1747,2.3635);
\coordinate (b0xc0) at (0.0000,2.2700);
\coordinate (b0xd0) at (-0.1747,2.3635);
\coordinate (c0xa0) at (0.1747,1.6365);
\coordinate (c0xb0) at (0.0000,1.7300);
\coordinate (c0xd0) at (-0.1747,1.6365);
\coordinate (d0xb0) at (-0.5453,2.1165);
\coordinate (d0xc0) at (-0.5453,1.8835);
\coordinate (d0xa2) at (-0.8250,1.8181);
\coordinate (a1xb1) at (1.5603,-1.5305);
\coordinate (a1xc1) at (1.3585,-1.4140);
\coordinate (a1xd2) at (1.1621,-1.6235);
\coordinate (b1xa1) at (1.9595,-1.3331);
\coordinate (b1xc1) at (1.9659,-1.1350);
\coordinate (b1xd1) at (2.1342,-1.0304);
\coordinate (c1xa1) at (1.3299,-0.9696);
\coordinate (c1xb1) at (1.4982,-0.8650);
\coordinate (c1xd1) at (1.5046,-0.6669);
\coordinate (d1xb1) at (2.1056,-0.5860);
\coordinate (d1xc1) at (1.9038,-0.4695);
\coordinate (d1xa0) at (1.9871,-0.1946);
\coordinate (a2xd0) at (-1.9871,-0.1946);
\coordinate (a2xb2) at (-2.1056,-0.5860);
\coordinate (a2xc2) at (-1.9038,-0.4695);
\coordinate (b2xa2) at (-2.1342,-1.0304);
\coordinate (b2xc2) at (-1.9659,-1.1350);
\coordinate (b2xd2) at (-1.9595,-1.3331);
\coordinate (c2xa2) at (-1.5046,-0.6669);
\coordinate (c2xb2) at (-1.4982,-0.8650);
\coordinate (c2xd2) at (-1.3299,-0.9696);
\coordinate (d2xb2) at (-1.5603,-1.5305);
\coordinate (d2xc2) at (-1.3585,-1.4140);
\coordinate (d2xa1) at (-1.1621,-1.6235);
\draw[line width=.45pt] (a0xb0) -- (b0xa0);
\draw[line width=.45pt] (a0xc0) -- (c0xa0);
\draw[line width=.45pt] (b0xc0) -- (c0xb0);
\draw[line width=.45pt] (b0xd0) -- (d0xb0);
\draw[line width=.45pt] (c0xd0) -- (d0xc0);
\draw[line width=.45pt] (d0xa2) -- (a2xd0);
\draw[line width=.45pt] (a1xb1) -- (b1xa1);
\draw[line width=.45pt] (a1xc1) -- (c1xa1);
\draw[line width=.45pt] (b1xc1) -- (c1xb1);
\draw[line width=.45pt] (b1xd1) -- (d1xb1);
\draw[line width=.45pt] (c1xd1) -- (d1xc1);
\draw[line width=.45pt] (d1xa0) -- (a0xd1);
\draw[line width=.45pt] (a2xb2) -- (b2xa2);
\draw[line width=.45pt] (a2xc2) -- (c2xa2);
\draw[line width=.45pt] (b2xc2) -- (c2xb2);
\draw[line width=.45pt] (b2xd2) -- (d2xb2);
\draw[line width=.45pt] (c2xd2) -- (d2xc2);
\draw[line width=.45pt] (d2xa1) -- (a1xd2);
\draw[line width=.55pt,fill=black!7] (a0xb0) -- (a0xc0) -- (a0xd1) -- cycle;
\node[vertex] at (a0xb0) {};
\node[vertex] at (a0xc0) {};
\node[vertex] at (a0xd1) {};
\draw[line width=.55pt,fill=black!7] (b0xa0) -- (b0xc0) -- (b0xd0) -- cycle;
\node[vertex] at (b0xa0) {};
\node[vertex] at (b0xc0) {};
\node[vertex] at (b0xd0) {};
\draw[line width=.55pt,fill=black!7] (c0xa0) -- (c0xb0) -- (c0xd0) -- cycle;
\node[vertex] at (c0xa0) {};
\node[vertex] at (c0xb0) {};
\node[vertex] at (c0xd0) {};
\draw[line width=.55pt,fill=black!7] (d0xb0) -- (d0xc0) -- (d0xa2) -- cycle;
\node[vertex] at (d0xb0) {};
\node[vertex] at (d0xc0) {};
\node[vertex] at (d0xa2) {};
\draw[line width=.55pt,fill=black!7] (a1xb1) -- (a1xc1) -- (a1xd2) -- cycle;
\node[vertex] at (a1xb1) {};
\node[vertex] at (a1xc1) {};
\node[vertex] at (a1xd2) {};
\draw[line width=.55pt,fill=black!7] (b1xa1) -- (b1xc1) -- (b1xd1) -- cycle;
\node[vertex] at (b1xa1) {};
\node[vertex] at (b1xc1) {};
\node[vertex] at (b1xd1) {};
\draw[line width=.55pt,fill=black!7] (c1xa1) -- (c1xb1) -- (c1xd1) -- cycle;
\node[vertex] at (c1xa1) {};
\node[vertex] at (c1xb1) {};
\node[vertex] at (c1xd1) {};
\draw[line width=.55pt,fill=black!7] (d1xb1) -- (d1xc1) -- (d1xa0) -- cycle;
\node[vertex] at (d1xb1) {};
\node[vertex] at (d1xc1) {};
\node[vertex] at (d1xa0) {};
\draw[line width=.55pt,fill=black!7] (a2xd0) -- (a2xb2) -- (a2xc2) -- cycle;
\node[vertex] at (a2xd0) {};
\node[vertex] at (a2xb2) {};
\node[vertex] at (a2xc2) {};
\draw[line width=.55pt,fill=black!7] (b2xa2) -- (b2xc2) -- (b2xd2) -- cycle;
\node[vertex] at (b2xa2) {};
\node[vertex] at (b2xc2) {};
\node[vertex] at (b2xd2) {};
\draw[line width=.55pt,fill=black!7] (c2xa2) -- (c2xb2) -- (c2xd2) -- cycle;
\node[vertex] at (c2xa2) {};
\node[vertex] at (c2xb2) {};
\node[vertex] at (c2xd2) {};
\draw[line width=.55pt,fill=black!7] (d2xb2) -- (d2xc2) -- (d2xa1) -- cycle;
\node[vertex] at (d2xb2) {};
\node[vertex] at (d2xc2) {};
\node[vertex] at (d2xa1) {};
\end{tikzpicture}
\par\small (b) The triangle expansion $T(R_3)$.
\end{minipage}
\caption{The 12-vertex ring of three diamonds and its 36-vertex triangle expansion. Each shaded triangle in (b) replaces one vertex of (a).}
\label{neg:fig:counterexample}
\end{figure}

\begin{proof}[Proof of Theorem~\ref{thm:mainB}]
Apply Lemma~\ref{common:expansion} and Theorem~\ref{neg:thm:rings}. Each base vertex gives three expanded vertices; cubicity then gives the edge count.
\end{proof}

Since the exceptional graph $\mathcal H=T(K_4)$ in~\cite{MEA} has 12 vertices, none of the graphs in Theorem~\ref{thm:mainB} is $\mathcal H$. Thus the family answers Problem~1 of that paper negatively. Its smallest member is $T(R_3)$, which has 36 vertices.

\section{Classification for two classes of bases}\label{sec:classification}
\subsection{Bridgeless claw-free bases}
We now characterise the triangle expansions whose bases are themselves claw-free. The bridgeless specialisation of Lemma~\ref{common:structure} is Oum's description~\cite[Proposition~1]{Oum}: a connected bridgeless claw-free cubic graph is $K_4$, a ring of diamonds, or a graph obtained from a bridgeless cubic multigraph $H$ by replacing every vertex with a triangle and some edges with strings of diamonds. To see why the skeleton is bridgeless, observe that a bridge of $H$ would give a bridge in the represented graph, whether its connection is a single edge or a string of diamonds. Moreover, $H$ is loopless: a loop at a cubic vertex forces its remaining incident edge to be a bridge.

\begin{theorem}\label{neg:thm:classification}
Let $G$ be a connected simple bridgeless claw-free cubic graph. The following statements are equivalent:
\begin{enumerate}[label=(\roman*)]
\item $T(G)$ admits a \packB-packing colouring;
\item $\PP(G)$ holds;
\item $G\not\cong K_4$ and $G$ is not a ring of an odd number of diamonds.
\end{enumerate}
\end{theorem}
\begin{proof}
The equivalence of (i) and (ii) is Theorem~\ref{neg:thm:reduction}. An independent set of $K_4$ has size at most one, and deleting it leaves a triangle. Hence $\PP(K_4)$ fails. The ring cases are settled by Theorem~\ref{neg:thm:rings}.

It remains to consider a graph in the third case of Oum's description. Let $H$ be its bridgeless cubic skeleton. By Petersen's theorem, $H$ has a perfect matching $M$. Orient each cycle of the $2$-factor $H-M$. A cycle consisting of two parallel edges is oriented by directing those two edges in opposite directions.

In the triangle replacing a vertex $v$ of $H$, denote by $p_v^-$, $p_v^+$ and $p_v^M$ the vertices corresponding to the incoming, outgoing and matching edges, respectively. Define $C$ by selecting $p_v^+$ in every such triangle and one centre in every inserted diamond.

This set is independent. Within each triangle or diamond only one vertex is selected. An unexpanded nonmatching edge joins an outgoing vertex to an incoming vertex, so only one endpoint is selected. An unexpanded matching edge joins two unselected matching vertices. Edges entering or joining diamonds have unselected diamond terminals at their ends, and create no further conflict.

Every selected diamond centre has the other centre as a private neighbour. For a selected triangle vertex $p_v^+$, use $p_v^M$ as a private neighbour. Its internal neighbours are $p_v^+$ and $p_v^-$; its external neighbour is either an unselected matching vertex or an unselected diamond terminal. Thus it has exactly one neighbour in $C$.

Finally, $G-C$ is a disjoint union of paths. Each vertex triangle retains the edge $p_v^-p_v^M$, and every diamond retains the path between its terminals through its unselected centre. Each matching edge, whether expanded or not, therefore connects two retained triangle edges to form a path. A nonmatching edge has lost its outgoing triangle endpoint. If it is unexpanded, it disappears; if it contains a diamond string, its remaining part is a path attached only at its incoming endpoint. Each of the paths associated with a matching edge acquires at most one such extension at either end. These components exhaust $G-C$. Hence $G-C$ is bipartite, and $C$ proves $\PP(G)$.
\end{proof}

The claw-free assumption in this theorem concerns the \emph{base} $G$. Although $T(G)$ is always claw-free, the theorem does not classify expansions of arbitrary bridgeless cubic bases.

\subsection{Double expansions and perfect matchings}
The private-neighbour condition also has a direct matching interpretation after a second triangle expansion. The next result does not require the absence of bridges.

\begin{theorem}\label{neg:thm:double}
Let $H$ be a loopless cubic multigraph. Then $T(T(H))$ admits a \packB-packing colouring if and only if $H$ has a perfect matching.
\end{theorem}
\begin{proof}
Put $G=T(H)$, which is simple by Lemma~\ref{common:expansion}. Suppose first that $\PP(G)$ holds, witnessed by $C$. Because $C$ is independent and $G-C$ is bipartite, $C$ meets each replacing triangle of $G$ in exactly one vertex. Say that a vertex $v$ of $H$ \emph{selects} the edge corresponding to that chosen incidence. No edge can be selected by both endpoints, since its two incidences are adjacent in $G$.

Writing $n=|V(H)|$, exactly $n$ distinct edges are selected, leaving
\[
|E(H)|-n=\frac n2
\]
unselected edges. The external neighbour of a chosen incidence cannot be private: it is adjacent both to that incidence and to the chosen incidence in its own triangle. Therefore every private neighbour must lie in the same triangle as the chosen vertex. Such an internal neighbour corresponds to an edge selected by neither endpoint. Consequently, every vertex of $H$ is incident with an unselected edge. The $n/2$ unselected edges cover all $n$ vertices, so they form a perfect matching.

Conversely, let $M$ be a perfect matching of $H$. Orient every cycle of $H-M$, and choose the outgoing incidence in each replacing triangle of $G$. These chosen vertices form an independent set $C$. The matching incidence in the same triangle is a private neighbour of each chosen vertex. After deleting $C$, every triangle retains its incoming--matching edge, every external nonmatching edge has lost its outgoing endpoint, and every external matching edge remains. Hence $G-C$ is a disjoint union of copies of $P_4$, one for each edge of $M$. In particular, $\PP(G)$ holds. Apply Theorem~\ref{neg:thm:reduction} in both directions.
\end{proof}

\begin{corollary}
Let $H$ be a bridgeless loopless cubic multigraph. Then $T(T(H))$ admits a \packB-packing colouring.
\end{corollary}
\begin{proof}
Petersen's theorem supplies a perfect matching of $H$.
\end{proof}

\section{Open questions}\label{sec:questions}
The family in Theorem~\ref{thm:mainB} provides counterexamples of order $12k$ for every odd $k\ge3$. It remains to determine the minimum order of a connected claw-free subcubic graph, distinct from $T(K_4)$, that admits no \packB-packing colouring.

Theorem~\ref{neg:thm:reduction} also suggests studying gated independent odd-cycle transversals in cubic graphs beyond the claw-free case. In particular, the odd rings of diamonds used here have cyclic edge cuts of size two. A natural next case is to characterise the simple cyclically $4$-edge-connected cubic graphs that admit such a transversal.

The general question of Gastineau and Togni~\cite{GT}, asking whether every connected subcubic graph other than the Petersen graph is $(1,1,2,3)$-packing colourable, remains open~\cite[Problem~2]{MEA}. The claw-free case is covered by the result of Mortada, El Zein and Al Hajjar cited in the introduction; the two results proved here concern the different sequences $\packA$ and $\packB$.

\end{document}